\documentclass[a4paper,11pt]{amsart}
\usepackage[latin1]{inputenc}
\usepackage[T1]{fontenc}
\usepackage{amsmath}
\usepackage{amsfonts}
\usepackage{amssymb}
\usepackage{amsthm}
\usepackage{mathrsfs}
\usepackage{enumitem}
\usepackage{graphicx}
\usepackage{footnote}
\usepackage{hyperref}

\renewcommand{\geq}{\geqslant}
\renewcommand{\leq}{\leqslant}
\newtheorem{thm}{Theorem}

\newtheorem{cor}[thm]{Corollary}
\newtheorem{prop}[thm]{Proposition}
\newtheorem{lem}[thm]{Lemma}

\usepackage{color}
\definecolor{darkgreen}{rgb}{0,0.4,0}
\definecolor{MyDarkBlue}{rgb}{0,0.08,0.50}
\definecolor{BrickRed}{rgb}{0.65,0.08,0}

\hypersetup{
colorlinks=true,       
    linkcolor=blue,          
    citecolor=red,        
    filecolor=BrickRed,      
    urlcolor=darkgreen        
}

\title[Walks reaching against all odds the other side of the~quarter~plane]{Random walks reaching against all odds the other side of the~quarter~plane}

\author{Johan S.H. van Leeuwaarden}
\address{Department of Mathematics and Computer Science, Eindhoven University of Technology,
P.O.\ Box 513, 5600 MB  Eindhoven, The Netherlands}
\email{j.s.h.v.leeuwaarden@tue.nl}

\author{Kilian Raschel}
\address{CNRS and Laboratoire de Math\'ematiques et Physique Th\'eorique, Universit\'e de Tours, Parc de Grandmont, 37200 Tours, France}
\email{Kilian.Raschel@lmpt.univ-tours.fr}

\keywords{Random walks in the quarter plane; Hitting probability of the boundaries; Nucleosome shifting; Voter model; Asymmetric exclusion process}
\subjclass{Primary 60G50; Secondary 82C22; 30E20}

 \date{\today}

\begin{document}

\begin{abstract}
For a homogeneous random walk in the quarter plane with nearest-neighbor transitions, starting from some state $(i_0,j_0)$, we study the event that  the walk reaches the vertical axis, before reaching the horizontal axis. We derive a certain integral representation for the probability of this event, and an asymptotic expression for the case when $i_0$ becomes large, a situation in which the event becomes highly unlikely.  The integral representation follows from the solution of a boundary value problem and is involves a conformal gluing function. The asymptotic expression follows from the asymptotic evaluation of this integral. Our results find applications in a model for nucleosome shifting, the voter model and the asymmetric exclusion process.
\end{abstract}

\maketitle

\section{Introduction}
Consider homogeneous random walks in the quarter plane with nearest-neighbor transitions. For such random walks, starting from some state $(i_0,j_0)$, we study the event of reaching the vertical axis, before reaching the horizontal axis. We derive an integral representation for the probability of this event, and  an asymptotic expression for the case when $i_0$ becomes large, a situation in which the event becomes highly unlikely. We use the classical method of solving for the generating function via functional equations and boundary value problems.

Our primary motivation is the work of Opheusden and Redig \cite{OpRe} on the following one-dimensional particle system. Consider three particles in $\mathbb{Z}$ and let $\eta_\ell(n)$ denote the position of particle $\ell$ after $n$ steps, with initial positions $\eta_1(0)<\eta_2(0)<\eta_3(0)$. The particles each get a weight and are then equipped with the following dynamics. At each time step, one of the particles is selected with probabilities proportional to their weights. The chosen particle is then moved to either the left or the right, with equal probability. Denote by $X(n)=\eta_2(n)-\eta_1(n)$ and $Y(n)=\eta_3(n)-\eta_2(n)$ the pairwise distances. The discrete-time Markov chain $(X(n),Y(n))_{n\in \mathbb{Z}_+}$, with $\mathbb{Z}_+=\{0,1,\ldots\}$, then clearly is a random walk in the quarter plane $\mathbb{Z}_+^2$. In particular, with particle two having weight $\nu$ and particles one and three weights $\lambda$, one obtains the walk in Figure \ref{Three_motivations}(a). For this walk Opheusden and Redig \cite{OpRe} studied the event of the Markov chain, starting from $(X(0),Y(0))=(14,1)$, reaching the vertical axis, before the horizontal axis. This event plays an important role in studying a nucleosome shifting with respect to the DNA sequence. In \cite{OpRe} an asymptotic expression was derived; see \eqref{oph}.

Another application of this work is the one-dimensional voter model. This lattice-based interacting particle system is used to model the spread of an opinion through a static population via nearest-neighbor interactions, and finds application in modeling competing species.
It is a discrete-time process on $\{0,1\}^{\mathbb{Z}}$, where each site in $\mathbb{Z}$ is labeled either $0$ or $1$. Two adjacent sites (a pair) are called an unlike pair when the labels are $01$ or $10$. The voter model then has the following dynamics.
At each time step the model selects uniformly at random from amongst all unlike pairs. The chosen pair is flipped to either $00$ or $11$, with equal chance of each.
\begin{figure}[t]
\begin{center}
\begin{picture}(000.00,780.00)
\hspace{-100mm}\hspace{-2mm}\includegraphics{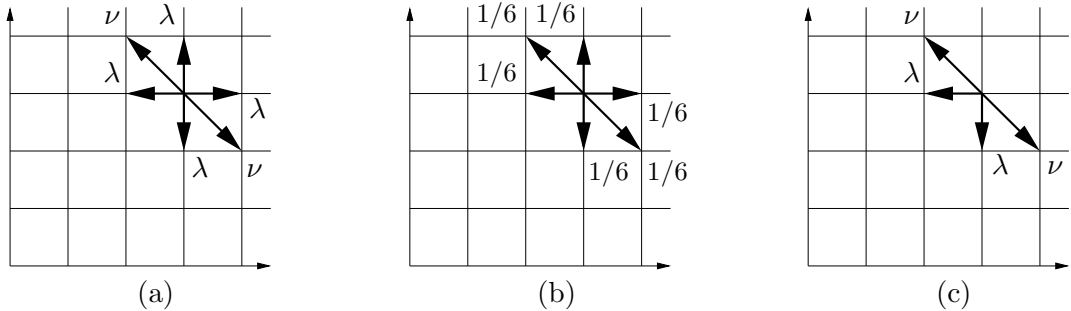}
\end{picture}
\end{center}
\vspace{-234mm}
\caption{Walks considered in \cite{OpRe}, \cite{Men01} and \cite{Godr}, respectively}
\label{Three_motivations}
\end{figure}
Since our version of the voter model lives on the infinite lattice, the ground state is the so-called Heaviside configuration
     \begin{equation*}
          \ldots 111000\ldots,
     \end{equation*}
and the initial configuration is assumed to be one with a finite number of unlike pairs.
For example,  with $N$ the number of finite blocks of zeros (or ones), a configuration looks like
\begin{equation*}
          \ldots 111
                 \overbrace{0000}^{\kappa_{1}}
                 \overbrace{1111}^{\sigma_{1}}
                 \overbrace{00}^{\kappa_{2}}
                 \overbrace{11111}^{\sigma_{2}}
          \ldots \overbrace{000000}^{\kappa_{N}}
                 \overbrace{1111}^{\sigma_{N}}
                 000
          \ldots ,
     \end{equation*}
with $\kappa_{\ell}$ (resp.\ $\sigma_{\ell}$) the size of the $\ell$th block of zeros (resp.\ ones). It is clear that $N$ is a non-increasing
function of time, since blocks will merge as time progresses. In fact, the Heaviside configuration ($N=0$) is an absorbing state.
Therefore, for the voter model, a crucial characteristic is the hitting time $\tau$ of the Heaviside configuration.
In \cite{Men01} it is shown that
$\mathbb{E}[\tau^{3/2-\epsilon}]<\infty$ and
$\mathbb{E}[\tau^{3/2+\epsilon}]=\infty$, for any
$\epsilon>0$ and any initial configuration. In proving the latter fact,
it suffices to consider the case $N=1$, because $N$ is non-increasing and hence the process always has to pass before absorption through $N=1$. Therefore, in \cite{Men01} the process
$(\kappa_{1}(n),\sigma_{1}(n))_{n\in \mathbb{Z}_+}=(X(n),Y(n))_{n\in \mathbb{Z}_+}$ is considered, which is clearly a random walk in the quarter plane
that is absorbed when it reaches
the boundary $\{(0,0)\} \cup \{ (i,0) : i\geq 1 \} \cup \{ (0,j ): j\geq 1 \}$. Define $\mathbb{Z}_+^*=\{1,2,\ldots\}$.
With
     \begin{equation*}
          p_{i,j}=\mathbb{P}\left[(X(n+1),Y(n+1))=(X(n),Y(n))+(i,j) \big|(X(n),Y(n))\in\mathbb{Z}_+^{* 2}\right],
     \end{equation*}
the dynamics of the voter model is described by  $p_{1,0}=p_{1,-1}=p_{0,-1}=p_{-1,0}=p_{-1,1}=p_{0,1}=1/6$, see Figure \ref{Three_motivations}(b). This random walk thus plays an important role in the voter model. It describes the situation in which the final two remaining groups try to impose their opinions on each other. This situation is in many cases rather persistent, particularly when both groups are of considerable size; see \cite{DuLe} for some simulation results that support this fact. We are interested in the situation in which one of the two groups forms a clear minority, and nevertheless, wins the battle with the other much larger group. It is clear that this is a large deviations event, and it is for this event that we obtain precise asymptotics.

A third application of the results in this paper is the phenomenon of spontaneous symmetry breaking in the asymmetric exclusion process.
Godr\`{e}che {\it et al.}~\cite{Godr} show that in some limiting regime, this phenomenon can be formulated as the hitting probability of the random walk with the transitions as in Figure \ref{Three_motivations}(c). Contrary to the first two applications, this random walk clearly has a negative drift.

Hence, in all three applications, we are interested in random walks reaching the vertical axis, before the horizontal axis, in situations where reaching first the horizontal axis is much more likely. In the next section we present our results for the probability of this event, for both zero-drift and negative-drift random walks. We shall also discuss the consequences for the three applications.

\section{Main results}
Denote by $(X,Y)=(X(n),Y(n))_{n\in\mathbb{Z}_+}$ a random walk in the quarter plane $\mathbb{Z}_+^2$, and let $\mathbb{P}_{(i_{0},j_{0})}[\mathscr{E}]$ be the probability of event $\mathscr{E}$ conditional on $(X(0),Y(0))=(i_{0},j_{0})$. Throughout we shall make the following assumption:
\begin{enumerate}[label=(H\arabic{*}),ref={\rm (H\arabic{*})}]
     \item \label{small_jumps} The walk is homogeneous inside of the quarter plane, with transition probabilities $\{p_{i,j}\}_{-1\leq i,j\leq 1}$ to the eight nearest neighbors.
\end{enumerate}
Denote the horizontal and vertical axes by
\begin{equation*}
\mathcal{H}=\{ (i,0) : i\geq 0 \}, \qquad \mathcal{V}=\{ (0,j ): j\geq 0 \},
\end{equation*}
and define $\mathcal{H}^*=\mathcal{H}\setminus \{(0,0)\}$ and $\mathcal{V}^*=\mathcal{V}\setminus \{(0,0)\}$. The principal object of study in this paper is the probability
     \begin{equation}
     \label{mainprob}
          \mathbb{P}_{(i_0,j_0)}[(X,Y) \textnormal{ hits } \mathcal{V} \textnormal{ before } \mathcal{H}^*],
     \end{equation}
for which we derive an exact expression, as well as an asymptotic expression for the large-deviations case $i_0\rightarrow\infty$.

In this paper we shall for the most part restrict to random walks $(X,Y)$ that, besides \ref{small_jumps}, satisfy the following assumptions:
\begin{enumerate}[label=(H\arabic{*}),ref={\rm (H\arabic{*})}]
\setcounter{enumi}{1}
     \item \label{non_degenerate} In the list $p_{1,1},p_{1,0},p_{1,-1},p_{0,-1},p_{-1,-1},p_{-1,0},p_{-1,1},p_{0,1}$,
                 there are no three~consecutive zeros;
     \item \label{non_degenerate2} $p_{1,1}+p_{-1,1}+p_{-1,-1}+p_{1,-1}<1$;
     \item \label{drift} The drifts are non-positive: $\sum_{-1\leq i,j\leq 1}i p_{i,j}\leq 0$ and $\sum_{-1\leq i,j\leq 1}j p_{i,j}\leq 0$.
\end{enumerate}

Assumption \ref{drift} guarantees that the random walk will hit one of the boundaries with probability one. With assumption \ref{small_jumps} and \ref{non_degenerate2}  we can use the general framework for random walks in the quarter plane developed by Fayolle {\it et al.}~\cite{FIM} (see Subsection \ref{App_exp_f} for more details). Assumption \ref{non_degenerate} excludes degenerate random walks, which can typically be analyzed using easier methods.

Here is our first main result.
\begin{thm}\label{main_theorem_zero}
Let $(X,Y)$ be a random walk satisfying \ref{small_jumps}--\ref{drift}. If
     \begin{equation}
     \label{drift_zero}
          \textstyle\sum_{-1\leq i,j\leq 1}i p_{i,j}=0\quad \text{and}\quad  \sum_{-1\leq i,j\leq 1}jp_{i,j}=0,
     \end{equation}
 there exists a constant $A\in(0,\infty)$ such that
     \begin{equation*}
          \mathbb{P}_{(i_0,j_0)}[(X,Y) \textnormal{ hits } \mathcal{V} \textnormal{ before } \mathcal{H}^*]\sim  A\frac{j_0}{i_0},
          \qquad i_0\to \infty.
     \end{equation*}
     \end{thm}
In Section \ref{Asymptotic} we present an explicit expression for the constant $A$ in Theorem \ref{main_theorem_zero}. For the model in Figure \ref{Three_motivations}(a) of nucleosome shifting, an inspection of Theorem \ref{main_theorem_zero} reveals
      \begin{equation}\label{oph}
          \mathbb{P}_{(i_0,j_0)}[(X,Y) \textnormal{ hits } \mathcal{V} \textnormal{ before } \mathcal{H}^*]\sim  \frac{\sqrt{1-\frac{\nu^2}{(\nu+\lambda)^2}}}{\arccos(\frac{\nu}{\nu+\lambda})}\frac{j_0}{i_0},
          \qquad i_0\to \infty.
     \end{equation}
Opheusden and Redig \cite{OpRe}    were able to derive \eqref{oph} using the following approach.
 First, define a generating function of which the probabilities in \eqref{mainprob} are the coefficients. Then show that this generating function satisfies a certain functional equation. This is the functional equation that is archetypal of random walks in the quarter plane, see \eqref{functional_equation}. The functional equation defines a characteristic curve, and by considering its tangent points one can determine the dominant singularity of the generating function. The nature of this dominant singularity then gives the asymptotic decay, in this case $O(1/i_0)$. Hence, an asymptotic estimate for \eqref{mainprob} is derived by studying a functional equation without having to solve it. The only drawback of this approach is that one cannot obtain  the constant term in the  asymptotic expression, because this would require an exact expression for the generating function and an investigation of this exact expression in the vicinity of its dominant singularity. Despite this fact, Opheusden en Redig were able to derive the constant term in \eqref{oph} by studying the continuum limit of the random walk. They conjectured that the asymptotic behavior of the continuum limit is the same as for the random walk, and provided strong numerical evidence. Here we provide the  proof of this conjecture.

For the voter model in Figure \ref{Three_motivations}(b), our Theorem \ref{main_theorem_zero} gives
         \begin{equation*}
          \mathbb{P}_{(i_0,j_0)}[(X,Y) \textnormal{ hits } \mathcal{V} \textnormal{ before } \mathcal{H}^*]\sim  \frac{3\sqrt{3}}{2\pi}\frac {j_0}{i_0},
          \qquad i_0\to \infty.
     \end{equation*}

We next present a result for random walks with a negative drift.
  \begin{thm}\label{main_theorem_non-zero}
Let $(X,Y)$ be a random walk satisfying \ref{small_jumps}--\ref{drift}. If
     \begin{equation}
     \label{drift_neg}
          \textstyle\sum_{-1\leq i,j\leq 1}i p_{i,j}\leq0 \quad \text{and} \quad  \sum_{-1\leq i,j\leq 1}jp_{i,j}<0,
     \end{equation}
     there exist constants $B(j_0)\in(0,\infty)$ and $\rho\in(0,1)$ such that
     \begin{equation*}
          \mathbb{P}_{(i_0,j_0)}[(X,Y) \textnormal{ hits } \mathcal{V} \textnormal{ before } \mathcal{H}^*]\sim  {B(j_0)}\frac{\rho^{i_0}}{i_0^{3/2}},
          \qquad i_0\to \infty.
     \end{equation*}
\end{thm}
The same result can be shown to hold for random walks with no transitions to the North, North-East and East. Introduce the assumption
\begin{enumerate}[label=(H\arabic{*}'),ref={\rm (H\arabic{*}')}]
     \setcounter{enumi}{1}
     \item \label{degenerate} $p_{-1,1}+p_{-1,0}+p_{-1,-1}+p_{0,-1}+p_{1,-1}=1$, $p_{-1,1}\neq 0$, $p_{1,-1}\neq 0$ and $p_{-1,1}+p_{1,-1}\neq 1$.
\end{enumerate}
Note that this assumption is satisfied by the random walk in Figure \ref{Three_motivations}(c), and that \ref{small_jumps} and \ref{degenerate} immediately render \ref{non_degenerate2} and \ref{drift}. We have the following result.
\begin{thm}
\label{theorem_Godreche}
Let $(X,Y)$ be a random walk satisfying \ref{small_jumps} and \ref{degenerate}. There exist constants $C(j_0)\in(0,\infty)$ and $\rho\in(0,1)$ such that
     \begin{equation*}
          \mathbb{P}_{(i_0,j_0)}[(X,Y) \textnormal{ hits } \mathcal{V} \textnormal{ before } \mathcal{H}^*]\sim  {C(j_0)}\frac{\rho^{i_0}}{i_0^{3/2}},
          \qquad i_0\to \infty.
     \end{equation*}
\end{thm}
Explicit expressions for the quantities $B(j_0)$, $C(j_0)$ and $\rho$ in Theorem \ref{main_theorem_non-zero} and Theorem \ref{theorem_Godreche} are derived in Section \ref{Asymptotic}. For the model of Godr\`{e}che {\it et al.}~\cite{Godr} in Figure \ref{Three_motivations}(c), Theorem \ref{theorem_Godreche} yields
          \begin{equation}\label{godd}
          \mathbb{P}_{(i_0,j_0)}[(X,Y) \textnormal{ hits } \mathcal{V} \textnormal{ before } \mathcal{H}^*]\sim  {C(j_0)}\frac{\rho^{i_0}}{i_0^{3/2}},
          \qquad i_0\to \infty,
     \end{equation}
     with
     \begin{equation*}
          \rho=\frac{2}{\lambda}[1-\lambda-\sqrt{(1-\lambda)(1-2\lambda)}].
     \end{equation*}
This result matches with Godr\`{e}che {\it et al.}~\cite[Equation (6.41)]{Godr}. As in  Opheusden and Redig \cite{OpRe}, Godr\`{e}che {\it et al.}~\cite{Godr} used a functional equation to derive the term $\rho^{i_0}{i_0^{-3/2}}$ in \eqref{godd}, but since the functional equation was not solved explicitly, it was impossible to derive the constant term  $C(j_0)$. We provide an exact expression for $C(j_0)$ in Lemma \ref{expression_B}.

\begin{figure}[t]
\begin{center}
\begin{picture}(000.00,780.00)
\hspace{-46mm}\hspace{-2mm}\includegraphics{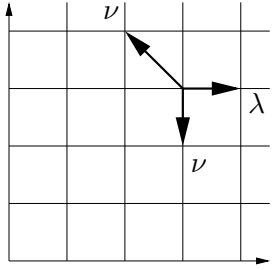}
\end{picture}
\end{center}
\vspace{-235mm}
\caption{A tandem queue.}
\label{Tandem}
\end{figure}

Let us finally present a result for random walks with a negative drift in the horizontal direction, and zero drift in the vertical direction.
\begin{thm}\label{main_theorem_non-zero_zero}
Let $(X,Y)$ be a random walk satisfying \ref{small_jumps}--\ref{drift}. If
     \begin{equation}
     \label{drift_neg_zero}
          \textstyle\sum_{-1\leq i,j\leq 1}i p_{i,j}<0 \quad \text{and} \quad  \sum_{-1\leq i,j\leq 1}jp_{i,j}=0,
     \end{equation}
     there exists a constant $D\in(0,\infty)$ such that
     \begin{equation*}
          \mathbb{P}_{(i_0,j_0)}[(X,Y) \textnormal{ hits } \mathcal{V} \textnormal{ before } \mathcal{H}^*]\sim  D \frac{j_0}{i_0^{1/2}},
          \qquad i_0\to \infty.
     \end{equation*}
\end{thm}
An example of such a walk is displayed in Figure \ref{Tandem}, with $p_{1,0}=\lambda$, $p_{0,-1}=\nu$ and $p_{-1,1}=\nu$,
assuming $\lambda<\nu$. This walk represents the transitions of a tandem queue with Poisson arrivals at queue 1 with rate $\lambda$, exponential services at both queues with mean $1/\nu$, and all customers traversing from the first queue to the second queue before leaving the system. In this case \eqref{mainprob} describes the probability that, starting with $i_0$ customers in queue 1, and $j_0$ customers in queue 2, queue 2 empties before queue 1. The constant $D$ is identified in Section \ref{Asymptotic} and for the tandem queue in Figure \ref{Tandem} takes the form
     \begin{equation*}
          D=\sqrt{\frac{\nu-\lambda}{\pi\nu}}.
     \end{equation*}

The remainder of the paper is structured as follows. We first derive, in Section \ref{explicit}, an explicit expression for the probability \eqref{mainprob}, for which we rely heavily on earlier work in \cite{FR2,KuRa}. In fact, this analysis leads to expressions for the generating functions,  of which the probabilities in \eqref{mainprob} are the coefficients, in terms of integrals that involve certain conformal gluing functions.
In Section \ref{Asymptotic} we prove Theorems \ref{main_theorem_zero}--\ref{main_theorem_non-zero_zero} by asymptotically evaluating the exact integral expressions for the generating functions obtained in Section \ref{explicit}. Finally, inspired by  \cite{OpRe}, we present the continuum limit for the zero-drift random walk in Section \ref{driftcon}. This continuum limit allows for an easy and explicit analysis of the type of event described in \eqref{mainprob}. We show that the constant that arises in the continuum limit matches with the exact constant we obtain in our precise asymptotic expression for \eqref{mainprob}. This seems to suggest an interchange-of-limits, but establishing a formal proof of this fact remains an open problem.

\section{Exact integral representations}
\label{explicit}
\setcounter{equation}{0}

We first derive an explicit expression for the probability \eqref{mainprob} in terms of integral representations for the generating functions
     \begin{align}
          h^{i_{0},j_{0}}(x)&=
          \sum_{i\geq 1} \mathbb{P}_{(i_0,j_0)}
          [(X,Y) \text{ hits } \mathcal{H} \textnormal{ before } \mathcal{V}^* \text{ and at } (i,0)]x^{i-1},
          \label{exp_h}\\
          \widetilde h^{i_{0},j_{0}}(y)&=
          \sum_{j\geq 1} \mathbb{P}_{(i_0,j_0)}
          [(X,Y) \text{ hits } \mathcal{V} \textnormal{ before } \mathcal{H}^* \text{ and at } (0,j)]y^{j-1}.
          \nonumber
          \end{align}
          Let
          \begin{align}
          h^{i_0,j_0}_{0,0}&=
          \mathbb{P}_{(i_0,j_0)}
          [(X,Y) \text{ hits } (0,0) \textnormal{ before } \mathcal{H}^* \cup \mathcal{V}^*]\nonumber
     \end{align}
and note that the probability \eqref{mainprob} follows from
     \begin{equation}
     \label{sufficient_tdgf}
          \mathbb{P}_{(i_0,j_0)}[(X,Y) \textnormal{ hits } \mathcal{V} \textnormal{ before } \mathcal{H}^*]
          =\widetilde h^{i_{0},j_{0}}(1)+h^{i_0,j_0}_{0,0}=1-h^{i_{0},j_{0}}(1).
     \end{equation}
The second equality in \eqref{sufficient_tdgf} is due to the fact that under \ref{small_jumps}--\ref{drift},  the process hits the boundary with probability $1$, see \cite{FMM}. In Subsection \ref{Notations} we introduce some classical notions regarding the framework in \cite{FIM} that aims at solving functional equations for the above generating functions using the theory of boundary value problems. This framework leads to the results presented in Subsection \ref{Exp}.

\subsection{Basic properties of the kernel}
\label{Notations}
A common and crucial quantity of interest in the study of walks with small steps (as in \ref{small_jumps}) in the quarter plane is the kernel
     \begin{equation}
     \label{def_K}
          K(x,y)= x y[ \textstyle\sum_{-1\leq i,j\leq 1}
          p_{i,j }x^{i} y^{j}  -1].
     \end{equation}
It can also be written as
     \begin{equation*}
          K(x,y) = a(x) y^{2}+ b(x) y + c(x) = \widetilde{a}(y) x^{2}+
          \widetilde{b}(y) x + \widetilde{c}(y),
     \end{equation*}
where
     \begin{equation}
     \label{def_a_b_c}
          \begin{array}{llllllllll}
               a(x) =& \hspace{-3mm}p_{1,1}x^{2}+& \hspace{-3mm}p_{0,1}x+& \hspace{-3mm}p_{-1,1},
               \ \ \ b(x) =& \hspace{-3mm}p_{1,0}x^{2}-& \hspace{-3mm}x+& \hspace{-3mm}p_{-1,0},
               \ \ \ c(x) =& \hspace{-3mm}p_{1,-1}x^{2}+& \hspace{-3mm}p_{0,-1}x+& \hspace{-3mm}p_{-1,-1},\\
               \widetilde{a}(y) =& \hspace{-3mm} p_{1,1}y^{2}+& \hspace{-3mm}p_{1,0}y+& \hspace{-3mm}p_{1,-1},
               \ \ \ \widetilde{b}(y) =& \hspace{-3mm}p_{0,1}y^2-& \hspace{-3mm}y+& \hspace{-3mm}p_{0,-1},
               \ \ \ \widetilde{c}(y) =& \hspace{-3mm}p_{-1,1}y^{2}+& \hspace{-3mm}p_{-1,0}y+& \hspace{-3mm}p_{-1,-1}.
               \end{array}
     \end{equation}
We define
     \begin{equation}
     \label{def_d}
          d(x)=b(x)^{2}-4a(x) c(x),
          \qquad \widetilde{d}(y)=
          \widetilde{b}(y)^{2}-4
          \widetilde{a}(y)\widetilde{c}(y).
     \end{equation}
Under \ref{small_jumps}--\ref{non_degenerate}, the polynomial $d$ has degree three or four, and  we denote its roots by  $\{x_\ell\}_{1\leq \ell\leq 4}$ with
     \begin{equation*}
          |x_1|\leq |x_2|\leq |x_3|\leq |x_4|,
     \end{equation*}
and $x_4=\infty$ if $d$ is a third-degree polynomial. One can easily see that $x_1\in(-1,1)$ and that $x_4\in(1,\infty) \cup \{\infty\}\cup (-\infty,-1]$. As for the roots $x_2$ and $x_3$, they are positive and such that $x_1<x_2\leq 1\leq x_3$. Furthermore, $x_2=1$ (resp.\ $x_3=1$) if and only if $\sum_{-1\leq i,j\leq 1}j p_{i,j}= 0$ (resp.\ $\sum_{-1\leq i,j\leq 1}i p_{i,j}= 0$). The polynomial $\widetilde d$ in \eqref{def_d} and its roots $\{y_\ell\}_{1\leq \ell\leq 4}$ satisfy similar properties. These facts, as well as Lemma \ref{d_posi_nega} below, are proved in \cite[Chapter 2]{FIM}.

\begin{lem}
\label{d_posi_nega}
The polynomial $d$  is positive on $(x_2,x_3)\cup (x_4,x_1)$  and negative on $(x_1,x_2)\cup (x_3,x_4)$. Similarly, the polynomial $\widetilde d$ is positive on  $(y_2,y_3)\cup (y_4,y_1)$ and negative on $(y_1,y_2)\cup (y_3,y_4)$.
\end{lem}

In what follows, we call $X(y)$ and $Y(x)$ the algebraic functions defined by $K(X(y),y)=0$ and $K(x,Y(x))=0$. With \eqref{def_K}--\eqref{def_d} we have
     \begin{equation}
     \label{expression_X_Y}
          X(y)=\frac{-\widetilde b(y)\pm \sqrt{\widetilde d(y)}}{2 \widetilde a(y)},
          \qquad Y(x)=\frac{-b(x)\pm \sqrt{d(x)}}{2 a(x)},
     \end{equation}
where above and throughout, we use the principal determination of the square root, defined on $\mathbb C\setminus (-\infty,0]$, see \cite[Chapter 4]{JS}. The functions $X(y)$ and $Y(x)$ both have two branches, called $X_0$, $X_1$ and $Y_0$, $Y_1$. Lemma \ref{d_posi_nega} asserts that they are meromorphic on $\mathbb{C}\setminus ([y_1,y_2]\cup [y_3,y_4])$ and $\mathbb{C}\setminus ([x_1,x_2]\cup [x_3,x_4])$, respectively.

In the non-zero-drift case, we fix notation letting $X_0(1)<X_1(1)$ and $Y_0(1)<Y_1(1)$ (see \cite[Lemma 2.3.4]{FIM}). Then we have on the whole of $\mathbb C$ (see \cite[Theorem 5.3.3]{FIM})
     \begin{equation}
     \label{ineq_whole}
          |X_0(y)|\leq |X_1(y)|, \qquad |Y_0(x)|\leq |Y_1(x)|.
     \end{equation}

In the zero-drift case, $X_0(1)=X_1(1)=Y_0(1)=Y_1(1)=1$, and the previous convention has to be changed: we then choose $\vert X_0(-1)\vert<\vert X_1(-1)\vert$ and $\vert Y_0(-1)\vert<\vert Y_1(-1)\vert$ (see \cite[Lemma 6.5.1]{FIM}). Then \eqref{ineq_whole} still holds on the whole of $\mathbb C$ (by a continuity argument).

Let us finally introduce
     \begin{equation}
     \label{def_mu}
          \mu_{j_{0}}(x) = \frac{1}{[2 a(x)]^{j_{0}}}
          \sum_{k=0}^{(j_{0}-1)/2} \binom{j_0}{2k+1}
          d(x)^{k} [ - b(x)]^{j_{0}-(2k+1)}.
     \end{equation}
This quantity appears in the expression of $h^{i_{0},j_{0}}(x)$ that we shall give in Theorem \ref{expression_h}. It is closely related to $Y(x)$: since $d(x)$ is non-positive for $x\in [x_1,x_2]\cup [x_3,x_4]$, see Lemma \ref{d_posi_nega}, the two branches \eqref{expression_X_Y} of $Y(x)$ are complex conjugates in these intervals. Expression \eqref{expression_X_Y} and some elementary calculations then yield
     \begin{equation}
     \label{before_mu}
          Y_0(x)^{j_0}-Y_1(x)^{j_0}=\pm 2 i \sqrt{-d(x)}\mu_{j_{0}}(x).
     \end{equation}
In order to determine the sign $\pm$ in \eqref{before_mu}, we have to specify whether $x\to [x_1,x_2]\cup[x_3,x_4]$ from above ($\downarrow$) or below ($\uparrow$): indeed, remember that the branches $Y_0(x)$ and $Y_1(x)$ are not meromorphic on $x\in[x_1,x_2]\cup[x_3,x_4]$. 
For instance, if $x\downarrow [x_1,x_2]$, we have $\pm = -$ in \eqref{before_mu}, and if $x\uparrow [x_1,x_2]$, we have $\pm = +$ (see \cite[Subsection 3.2]{KuRa}).

\subsection{A crucial conformal mapping}
\label{Exp}
Before deriving an expression for the generating function $h^{i_0,j_0}$ defined in \eqref{exp_h}, we first need to introduce a certain conformal mapping. For this, define the closed curve in the complex plane
     \begin{equation*}
          X([y_1,y_2])=X_0([y_1,y_2])\cup X_1([y_1,y_2]),
     \end{equation*}
which is symmetrical with respect to the real axis (since for $y\in [y_1,y_2]$, $X_0(y)$ and $X_1(y)$ are complex conjugates, see Lemma \ref{d_posi_nega} and \eqref{expression_X_Y}) and goes around the segment $[x_1,x_2]$ (see \cite[Theorem 5.3.3]{FIM}). Denote by
     \begin{equation*}
          \mathscr{G}X([y_1,y_2])
     \end{equation*}
the set surrounded by $X([y_1,y_2])$, which in addition contains $[x_1,x_2]$. For instance, in the case of the simple random walk (with $p_{1,0}=p_{0,1}=p_{-1,0}=p_{0,-1}=1/4$), $X([y_1,y_2])$ is the unit circle (see \cite[Theorem 5.3.3]{FIM}), hence $\mathscr{G}X([y_1,y_2])$ is the unit disc, since $-1\leq x_1,x_2\leq 1$ (see Subsection \ref{Notations}).

Let us now introduce a conformal gluing function for the set $\mathscr{G}X([y_1,y_2])$, i.e., a function $w$ such that
     \begin{enumerate}
          \item \label{w-mero} $w$ is meromorphic in $\mathscr{G}X([y_1,y_2])$, and for $t$ on the boundary of the latter domain, $\lim_{x\to t}w(x)$ exists (in $\mathbb C\cup \{\infty\}$), provided $x\in\mathscr{G}X([y_1,y_2])$;
          \item \label{w-conformal} $w$ establishes a conformal mapping of $\mathscr{G}X([y_1,y_2])$ onto the complex plane $\mathbb{C}$ cut along an interval;
          \item \label{w-boundary} For all $t$ on the boundary of $\mathscr{G}X([y_1,y_2])$, i.e., for all $t$ in $X([y_1,y_2])$, $w(t)=w(\overline{t})$.
     \end{enumerate}
\subsubsection*{Examples}
If the domain $\mathscr{G}X([y_1,y_2])$ has a simple shape (for example, if it is bounded by a circle, or by an ellipse), then it is an easy task to determine a suitable conformal mapping.

A first example concerns the case where $\mathscr{G}X([y_1,y_2])$ is the unit disc (this is the case for the simple random walk, see above). Then the function
     \begin{equation*}
          w(t)=\frac{t}{(t-1)^{2}}
     \end{equation*}
is a suitable conformal gluing function. Indeed, \ref{w-mero} is obvious; further, $w$ is a conformal mapping from the unit disc to $\mathbb C\setminus (-1/4,\infty)$, so \ref{w-conformal} is satisfied; finally, for $t=\exp(i\phi)$ we have $w(t)=-1/(4\sin(\phi)^2)=w(\overline{t})$, and \ref{w-boundary} follows.

Another example is when the domain $\mathscr{G}X([y_1,y_2])$ can be bounded by an ellipse (this is the case for the walks, which, as in \cite{Godr},  satisfy \ref{degenerate}). Then expressions for $w$ can then be found in the literature (see, e.g., \cite[Theorem 6.3.1]{FIM}).

\subsubsection*{On the existence and the uniqueness of $w$}
The existence (without any explicit expression) of functions $w$ satisfying \ref{w-mero}, \ref{w-conformal} and \ref{w-boundary} follows from general results on conformal mappings (see \cite[Chapter 2]{LIT}). Any linear transformation $(\alpha w+\beta)/(\gamma w+\delta)$ of $w$ is also a suitable conformal mapping. In particular, there is no uniqueness of $w$. However, there is uniqueness of $w$ up to these linear transformations (see \cite[Chapter 2]{LIT}).

\subsubsection*{On explicit expressions for $w$}
In most cases, finding an explicit expression for $w$ turns out to be challenging. However, for the class of walks at hand, expressions for $w$ are available in the literature. In the non-zero-drift case, an expression for $w$ in terms of certain elliptic functions is obtained in \cite[Section 4]{KuRa}. In the zero-drift case, an expression for $w$ is found in \cite[Section 2]{FR2} (it is recalled here in \eqref{def_CGF_w}), by considering the zero-drift case as the limiting case of the non-zero-drift case. We refer to \cite{FR2,KuRa,Ra} for additional details.

\subsection{Exact hitting probabilities}
\label{Exp-2}
Denote by $w$ a conformal gluing function as in Subsection \ref{Exp}. With the notation introduced in Subsection \ref{Notations}, we have the following result.
\begin{thm}
\label{expression_h}
For $x\in \mathbb{C}\setminus (x_3,x_4)$,
     \begin{equation*}
          h^{i_{0},j_{0}}(x)=x^{i_{0}}Y_{0}(x)^{j_{0}}+
          \frac{1}{\pi}\int_{x_{1}}^{x_2} t^{i_{0}}\mu_{j_{0}}(t)
          \bigg[ \frac{w'(t)}{w(t)-w(x)}-\frac{w'(t)}{w(t)-w(0)}
          \bigg]\sqrt{-d(t)}\textnormal{d}t.
     \end{equation*}
\end{thm}
The proof of Theorem \ref{expression_h} is given in Subsection \ref{App_exp_f}.
Theorem \ref{expression_h}, together with the fact that $Y_{0}(1)=1$ (see \cite[Equation (5.3.2)]{FIM}), yields
     \begin{equation}
     \label{expression_h(1)}
          h^{i_{0},j_{0}}(1)=1+
          \frac{1}{\pi}\int_{x_{1}}^{x_2} t^{i_{0}}\mu_{j_{0}}(t)
          \bigg[ \frac{w'(t)}{w(t)-w(1)}-\frac{w'(t)}{w(t)-w(0)}
          \bigg]\sqrt{-d(t)}\textnormal{d}t.
     \end{equation}
Thanks to \eqref{sufficient_tdgf}, \eqref{expression_h(1)} immediately yields an expression for the hitting probability \eqref{mainprob}.

\subsection{Proof of Theorem \ref{expression_h}}
\label{App_exp_f}

We only sketch the proof of Theorem \ref{expression_h}, because we largely mimic the proof in \cite{KuRa} for the case of two positive drifts, i.e.,
     \begin{equation*}
          \textstyle\sum_{-1\leq i,j\leq 1}i p_{i,j}> 0,\qquad\sum_{-1\leq i,j\leq 1}j p_{i,j}> 0.
     \end{equation*}
A close examination of the proof in \cite{KuRa} makes clear that the result for positive drifts remains to hold in the zero-drift and negative-drift cases. The only difference between these cases is that the conformal gluing function $w$ introduced in Subsection \ref{Exp} will be different. To be somewhat more specific, we now present the four main steps of the proof of Theorem \ref{expression_h}, closely following the original approaches of \cite{FIM} and \cite{KuRa}.

\medskip

\textit{Step 0.} Using simple recursion relations it can be shown that $h^{i_{0},j_{0}}(x)$, $\widetilde h^{i_{0},j_{0}}(y)$ and $h_{0,0}^{i_{0},j_{0}}$ satisfy the functional equation (see \cite[Section 2]{KuRa})
     \begin{multline}
     \label{functional_equation}
          h^{i_{0},j_{0}}(x)+\widetilde h^{i_{0},j_{0}}(y)+h_{0,0}^{i_{0},j_{0}}-x^{i_0}y^{j_0}=K(x,y)\times\\
          \sum_{i,j\geq 1}\sum_{n\geq 0}\mathbb{P}_{(i_0,j_0)}[(X(n),Y(n))=(i,j),\ (X,Y) \text{ did not hit } \mathcal{H}\cup\mathcal{V} \text{ between } 0 \text{ and } n].
     \end{multline}
\medskip

\textit{Step 1.} Thanks to the fundamental identity \eqref{functional_equation}, we prove that $h^{i_0,j_0}(x)$ satisfies the following boundary value
problem. Let $\mathscr{G}X([y_1,y_2])$ be the set introduced in Subsection \ref{Exp}. Then
     \begin{enumerate}
          \item \label{class_function} $h^{i_0,j_0}$ is holomorphic in $\mathscr{G}X([y_1,y_2])$;
          \item \label{boundary_condition} For all $t$ on the boundary of $\mathscr{G}X([y_1,y_2])$,
          \begin{equation*}
               h^{i_0,j_0}(t)-h^{i_0,j_0}(\overline{t})=t^{i_0}Y_0(t)^{j_{0}}-\overline{t}{}^{i_0}Y_0(\overline{t})^{j_{0}}.
          \end{equation*}
     \end{enumerate}
The problem of finding a function satisfying \ref{class_function}--\ref{boundary_condition} is a particular instance of a boundary value problem with shift (the complex conjugation plays in \ref{boundary_condition} the role of the shift), see \cite{LIT} for an extensive treatment of this topic. Items \ref{class_function} and \ref{boundary_condition} follow from \cite[Theorem 6.5.2]{FIM}. Note that \ref{boundary_condition} is easily proved: it essentially suffices to evaluate \eqref{functional_equation} both at $X_0(y)$ and $X_1(y)$, for all $y\in[y_1,y_1]$. In this way, the kernel $K(x,y)$ vanishes, and in fact the right-hand side of \eqref{functional_equation} too. Finally, taking the difference of the equations corresponding to $X_0(y)$ and $X_1(y)$ leads to \ref{boundary_condition}.

\medskip

\textit{Step 2.} We transform the problem \ref{class_function}--\ref{boundary_condition} into a boundary value problem with a boundary condition on a segment. This can be done via a conformal gluing function for the set $\mathscr{G}X([y_1,y_2])$ as discussed in Subsection \ref{Exp}. We refer to \cite[Section 3]{Ra} for more details.

\medskip

\textit{Step 3.} The solution of the latter boundary value problem is elementary, see \cite{LIT} or \cite[Section 3]{Ra}, and  can be formulated in terms of Cauchy integrals. The explicit integral representation of $h^{i_0,j_0}(x)$ follows. A similar expression can be obtained for $\widetilde h^{i_0,j_0}(y)$.

\medskip

{\noindent {\bf Remark.}}
We can now elaborate on the reasons for assuming \ref{small_jumps} and \ref{non_degenerate2}. First, if we would allow larger jumps, we could still obtain a functional equation for the generating function of the hitting probabilities, but the technique that is used to solve \eqref{functional_equation} does not carry over. Further, finding a conformal gluing function $w$ requires introducing the Riemann surface defined by
     \begin{equation*}
          \{(x,y)\in\mathbb{C}^2: K(x,y)=0\}.
     \end{equation*}
 For small jumps as in \ref{small_jumps}, this Riemann surface has genus $0$ or $1$, see \cite{FIM}. If the jumps are larger, the genus of this Riemann surface increases, and the problem of finding $w$ becomes more intricate. The reason for assuming \ref{non_degenerate2} is that for $p_{1,1}+p_{-1,1}+p_{-1,-1}+p_{1,-1}=1$ it is possible that the branches \eqref{expression_X_Y} of the algebraic functions $X(y)$ and $Y(x)$ are meromorphic on the whole $\mathbb{C}$. In this case we cannot state (and solve!) a boundary value problem, since this requires complex conjugate branches on some interval.

\section{Asymptotic analysis}\label{Asymptotic}
\setcounter{equation}{0}

We now prove Theorems \ref{main_theorem_zero}--\ref{main_theorem_non-zero_zero}, by asymptotically evaluating the integral expressions of the generating function $h^{i_{0},j_{0}}$ derived in Section \ref{explicit}.

\subsection{Proof of Theorem \ref{main_theorem_zero} (zero-drift case)}
Assume \eqref{drift_zero} and let $a$ be defined as in \eqref{def_a_b_c}, $d$ as in \eqref{def_d}, and
     \begin{equation}
     \label{exp_theta}
          \theta=\arccos \bigg(-\frac{\sum_{-1\leq i,j\leq 1}i j p_{i,j}}
          {[\sum_{-1\leq i,j\leq 1}i^2 p_{i,j}]^{1/2}\cdot [\sum_{-1\leq i,j\leq 1}j^2 p_{i,j}]^{1/2}}\bigg).
     \end{equation}

\begin{lem}
\label{expression_C}
The constant $A$ in Theorem \ref{main_theorem_zero}  is given by
     \begin{equation}
     \label{eff_expression_C}
          \frac{[-d''(1)]^{1/2}}{2^{3/2}\theta a(1)}.
     \end{equation}
\end{lem}
In order to prove Theorem \ref{main_theorem_zero} and Lemma \ref{expression_C}, we first identify an appropriate conformal gluing function $w$ for the domain $\mathscr{G}X([y_1,y_2])$. Define
     \begin{equation}
     \label{def_f}
          f(t)=\left\{\begin{array}{lll}
          \displaystyle d''(x_{4})/6+d'(x_{4})/[t-x_{4}]& \text{if} & x_{4}\neq \infty,\\
          \displaystyle d''(0)/6+d'''(0)t/6 \phantom{{1^1}^{1}}& \text{if} & x_{4}=\infty.\end{array}\right.
     \end{equation}

The next two lemmas are taken from \cite[Section 2]{FR2}.
\begin{lem}
Let $f$ be as in \eqref{def_f}. In the zero-drift case \eqref{drift_zero}, the function $w$ defined by
\begin{equation}
     \label{def_CGF_w}
          w(t)=\sin\left(\frac{\pi}{\theta}
          \left[\arcsin\left\{\left[\frac{1}{3}-\frac{2f(t)}{d''(1)}\right]^{-1/2}\right\}-\frac{\pi}{2}\right]\right)^2
     \end{equation}
is a suitable conformal gluing function for the set $\mathscr{G}X([y_1,y_2])$.
\end{lem}
This exact expression renders the behavior of $w$ near $t=1$.
     \begin{lem}
     Let $w$ be as in \eqref{def_CGF_w} and $\theta$ as in \eqref{exp_theta}. There exists an $\alpha\neq 0$ such that
     \begin{equation}
     \label{behavior_chi}
          w(t)=\frac{\alpha +o(1)}{(1-t)^{\pi/\theta}}, \qquad t\uparrow1.
     \end{equation}
     \end{lem}

\begin{proof}[Proof of Theorem \ref{main_theorem_zero} and Lemma {\rm\ref{expression_C}}]
Thanks to \eqref{sufficient_tdgf}, it is enough to prove that
     \begin{equation*}
          h^{i_{0},j_{0}}(1)=1-A j_0/i_0+O(1/i_0^2)
     \end{equation*}
with $A$ as in \eqref{eff_expression_C}. First, since $w(1)=\infty$ (see \eqref{behavior_chi}) and since $x_2=1$ (see Subsection \ref{Notations}), \eqref{expression_h(1)} yields
     \begin{equation}
     \label{after_w(1)_infty}
          h^{i_{0},j_{0}}(1)=1-
          \frac{1}{\pi}\int_{x_{1}}^{1} t^{i_{0}}\mu_{j_{0}}(t)
          \frac{w'(t)}{w(t)-w(0)}\sqrt{-d(t)}\textnormal{d}t.
     \end{equation}
For $t\in[x_1,1]$, we introduce
     \begin{equation}
     \label{alpha_beta}
          \mu_{j_0}(t)\sqrt{-d(t)}=\sum_{k\geq 1} \alpha_k (t-1)^k,
          \qquad \frac{w'(t)}{w(t)-w(0)}= \sum_{k\geq-1} \beta_k (t-1)^{k}.
     \end{equation}
The first sum in \eqref{alpha_beta} starts at $k=1$  because $1$ is a double root of $d$. The second sum in \eqref{alpha_beta} starts at $k=-1$ because $w$ has a singularity at $1$ of the kind \eqref{behavior_chi}. With \eqref{after_w(1)_infty} we then obtain
     \begin{equation}
     \label{expression_h1}
          h^{i_{0},j_{0}}(1) = 1
          - \frac{1}{\pi} \int_{x_1}^{1} t^{i_{0}}
          [\alpha_1 \beta_{-1}+(\alpha_2 \beta_{-1}+\alpha_1 \beta_0)(t-1)+\ldots ]
          \,\text{d}t.
     \end{equation}
Since, for $p\geq 0$,
     \begin{equation*}
          \int_{x_1}^{1}t^{i_{0}} (t-1)^p \,\text{d} t =
          \frac{(-1)^{p} p!}{{i_0^{1+p}}}+O(i_0^{-2-p}),
     \end{equation*}
we deduce that
     \begin{equation}
     \label{ded}
          h^{i_{0},j_{0}}(1) = 1- \frac{\alpha_1 \beta_{-1}}{\pi i_0}+O(i_0^{-2}).
     \end{equation}
It remains to identify $\alpha_1$ and $\beta_{-1}$. First, since $w$ has a singularity of order $\pi/\theta$ at $1$ (see \eqref{behavior_chi}), it is immediate that $\beta_{-1}=-\pi/\theta$. In addition, since $1$ is a double root of $d$, we obtain that $\alpha_1=\mu_{j_{0}}(1) [-d''(1)/2]^{1/2}$. Moreover, the equality $d(1)=0$ together with \eqref{def_mu} implies that
     \begin{equation*}
          \mu_{j_{0}}(1)=\frac{j_0 [-b(1)]^{j_{0}-1}}{[2a(1)]^{j_{0}}}
                        =\frac{j_0}{2a(1)}\left[\frac{c(1)}{a(1)}\right]^{(j_{0}-1)/2}.
     \end{equation*}
The last identity follows from $-b(1)=2[a(1)c(1)]^{1/2}$, which indeed holds because $d(1)=0$ and $b(1)<0$. Moreover, under assumption \eqref{drift_zero} we have $a(1)=c(1)$, in such a way that $\alpha_1=[j_0/(2a(1))] \cdot [-d''(1)/2]^{1/2}$. The proof is completed.
\end{proof}

{\noindent {\bf Remark.}}
Theorem \ref{main_theorem_zero} provides first-order expansions of $h^{i_{0},j_{0}}(1)$ and $\widetilde h^{i_{0},j_{0}}(1)$.~By~extending our approach, we could obtain expansions up to any order, see \eqref{after_w(1)_infty}--\eqref{ded}.

\medskip

{\noindent {\bf Remark.}}
Notice that $d''(1)<0$, so that $A>0$, see \eqref{eff_expression_C}. Indeed, it is proved in \cite{FIM} that under \ref{small_jumps}--\ref{drift} and \eqref{drift_zero}, only two roots of $d$ are equal to $1$. In particular, $d''(1)\neq 0$. By continuity of $d''(1)$ with respect to  the parameters $\{p_{i,j}\}_{-1\leq i,j\leq 1}$, it is enough to check that for {one} walk, we have $d''(1)<0$. This can be easily done, for instance for the simple random walk. 

%

\subsection{Proofs of Theorems \ref{main_theorem_non-zero}--\ref{main_theorem_non-zero_zero} (negative-drift case)}
\label{proof_second_part_main_theorem}

\begin{lem}
\label{expression_B}
The constants $B(j_0)$ and $\rho$ in Theorem \ref{main_theorem_non-zero} are given by
      $\rho=x_2$ and
     \begin{equation}
     \label{def_Bj0}
          B(j_0)={\frac{x_2^{3/2}}{\sqrt{2\pi}}}\frac{j_0}{2a(x_2)}\left[\frac{c(x_2)}{a(x_2)}\right]^{(j_{0}-1)/2}d'(x_2)^{1/2}\beta_0,
     \end{equation}
with $\beta_0$ as in \eqref{def_beta_0}.
\end{lem}

\begin{proof}[Proof of Theorem \ref{main_theorem_non-zero} and Lemma \ref{expression_B}]
We now assume \eqref{drift_neg}. For $t\in[x_1,x_2]$ define
     \begin{equation}
     \label{alpha_beta_non-zero}
          \mu_{j_0}(t)\sqrt{-d(t)}=\sqrt{x_2-t}\sum_{k\geq 0} \alpha_k (t-x_2)^k,
          \quad \frac{w'(t)}{w(t)-w(1)}-\frac{w'(t)}{w(t)-w(0)}= \sum_{k\geq 0} \beta_k (t-x_2)^{k}.
     \end{equation}
With \eqref{expression_h(1)} we then obtain
     \begin{equation*}
          h^{i_{0},j_{0}}(1) = 1+ \frac{1}{\pi} \int_{x_1}^{x_2} t^{i_{0}}
          [\alpha_0 \beta_{0}+(\alpha_1 \beta_{0}+\alpha_0 \beta_1)(t-x_2)+\ldots ] \sqrt{x_2-t}
          \,\text{d}t.
     \end{equation*}
Since, for $p\geq 0$,
     \begin{equation}
     \label{asymp_integral_non-zero}
          \int_{x_1}^{x_2}t^{i_{0}} (x_2-t)^{1/2+p} \,\text{d} t =
          \Gamma(p+3/2)\frac{x_2^{3/2+i_0+p}}{i_0^{3/2+p}}+O(i_0^{-5/2-p}),
     \end{equation}
we deduce that
     \begin{equation}
     \label{ded_non-zero}
          h^{i_{0},j_{0}}(1) = 1+ \frac{\alpha_0 \beta_{0}\Gamma(3/2)x_2^{3/2+i_0}}{\pi i_0^{3/2}}+O(i_0^{-5/2}).
     \end{equation}
In particular, the fact that $\rho=x_2\in(0,1)$ is a direct consequence of \eqref{drift_neg} and \cite[Lemma 2.3.9]{FIM}.
It remains to identify $\alpha_0$ and $\beta_{0}$. First, we obviously have $\alpha_0= d'(x_2)^{1/2}\mu_{j_{0}}(x_2)$, see \eqref{alpha_beta_non-zero}. Further, the equality $d(x_2)=0$ together with \eqref{def_mu} gives
     \begin{equation*}
          \mu_{j_{0}}(x_2)=\frac{j_0 [-b(x_2)]^{j_{0}-1}}{[2a(x_2)]^{j_{0}}}
                          =\frac{j_0}{2a(x_2)}\left[\frac{c(x_2)}{a(x_2)}\right]^{(j_{0}-1)/2}.
     \end{equation*}
The analysis of $\beta_0$ is more elaborate and requires, as can be seen from  \eqref{alpha_beta_non-zero}, a detailed description of the conformal gluing function $w$ for the set $\mathscr{G}X([y_1,y_2])$. We shall make crucial use of a  conformal gluing function derived in \cite{FIM,KuRa}. Because the description of this conformal gluing function would require the introduction of many new symbols, we choose to just give some of its most important properties, and we refer to \cite[Equation (16)]{KuRa} for its full expression. The properties of $w$ we shall use here are the following (see \cite[Proposition 15]{KuRa}):
     \begin{enumerate}[label={\rm (P\arabic{*})},ref={\rm (P\arabic{*})}]
          \item \label{pole_x_2} $w$ has a simple pole at $x_2$;
          \item \label{other_poles} The other possible poles of $w$ are on $(x_2,x_3)\cap (X(y_2),\infty)$;
          \item \label{simple_contour} The set $w(X_\ell([y_1,y_2]))$ is a real interval without double points.
     \end{enumerate}
Properties \ref{pole_x_2} and \ref{other_poles} together with \eqref{alpha_beta_non-zero} imply that
     \begin{equation}
     \label{def_beta_0}
          \beta_0=\frac{w(0)-w(1)}{\lim_{t\to x_2}[(t-x_2)w(t)]}.
     \end{equation}
This quantity is well defined (i.e., finite): thanks to \ref{pole_x_2}, the denominator of \eqref{def_beta_0} is non-zero, and thanks to \ref{other_poles} and since $0<x_2$ and $1\leq X_\ell(y_2)$ (see Lemma \ref{location_X(y_2)}), $w(0)$ and $w(1)$ are finite. To conclude, let us show that $\beta_0$ is non-zero. For this, we first note that both $0$ and $1$ belong to the closure of $\mathscr{G}X([y_1,y_2])$.  Indeed, $X_\ell (y_1)\leq 0$, see \cite[Lemma 23]{KuRa}, and $X_\ell(y_2)\geq 1$, see Lemma \ref{location_X(y_2)}. If $0$ and $1$ are in the open domain $\mathscr{G}X([y_1,y_2])$, then $w(0)\neq w(1)$: indeed, by definition (see Subsection \ref{Exp}), $w$ is one-to-one in $\mathscr{G}X([y_1,y_2])$. If either $0$ or $1$ lies on the boundary of $\mathscr{G}X([y_1,y_2])$, the latter reasoning still works. If both $0$ and $1$ lie on the boundary and if $w(0)=w(1)$, \ref{simple_contour} gives $w(X_\ell([y_1,y_2]))=\mathbb{R}\cup\{\infty\}$. The latter is a contradiction with the connectedness of $\mathscr{G}X([y_1,y_2])$.
\end{proof}

\begin{proof}[Proof of Theorem \ref{main_theorem_non-zero_zero}] In this proof we assume \eqref{drift_neg_zero}. In particular, this implies that $x_2=1$ and $x_3>1$, see \cite[Lemma 2.3.9]{FIM}. Moreover, \ref{pole_x_2} then gives that the function $w$ has a simple pole at $1$, and the identity \eqref{after_w(1)_infty} still holds. For $t\in[x_1,1]$, we introduce
     \begin{equation}
     \label{alpha_beta_non-zero_zero}
          \mu_{j_0}(t)\sqrt{-d(t)}=\sqrt{1-t}\sum_{k\geq 0} \alpha_k (t-1)^k,
          \qquad \frac{w'(t)}{w(t)-w(0)}= \sum_{k\geq-1} \beta_k (t-1)^{k}.
     \end{equation}
Equations \eqref{after_w(1)_infty} and \eqref{alpha_beta_non-zero_zero} yield that
     \begin{equation*}
          h^{i_{0},j_{0}}(1) = 1
          - \frac{1}{\pi} \int_{x_1}^{1} t^{i_{0}}
          [\alpha_0 \beta_{-1}+(\alpha_1 \beta_{-1}+\alpha_0 \beta_0)(t-1)+\ldots ]
          \frac{1}{\sqrt{1-t}}\,\text{d}t.
     \end{equation*}
With \eqref{asymp_integral_non-zero}, we obtain
     \begin{equation*}
          h^{i_{0},j_{0}}(1) = 1-\frac{\Gamma(1/2)\alpha_0\beta_{-1}}{\pi i_0^{1/2}}+O(i_0^{-3/2}).
     \end{equation*}
Furthermore, the fact that $w$ has a simple pole at $1$ implies that $\beta_{-1}=-1$, and
     \begin{equation*}
          \alpha_0 = d'(1)^{1/2}\mu_{j_{0}}(1) =  j_0 \frac{d'(1)^{1/2}}{2a(1)}.
     \end{equation*}
In particular, the constant $D$ in Theorem \ref{main_theorem_non-zero_zero} is given by
     $
          D=d'(1)^{1/2}/(2\sqrt{\pi} a(1)).
     $
\end{proof}

\begin{lem}
\label{location_X(y_2)}
The quantity $X_\ell(y_2)$ is such that:
     \begin{itemize}
     \item If $\sum_{-1\leq i,j\leq 1}i p_{i,j}<0$, then $X_\ell(y_2)>1$;
     \item If $\sum_{-1\leq i,j\leq 1}i p_{i,j}=0$, then $X_\ell(y_2)=1$.
     \end{itemize}
\end{lem}

\begin{proof}
If $\sum_{-1\leq i,j\leq 1}i p_{i,j}=0$ we have $y_2=1$, see \cite[Lemma 2.3.9]{FIM}, and therefore $X_\ell(y_2)=$ $X_\ell(1)=1$, see \cite[Lemma 5.3.1]{FIM}. If we now assume that $\sum_{-1\leq i,j\leq 1}i p_{i,j}<0$ and $X_\ell(y_2)=1$, then thanks to \eqref{expression_X_Y} we obtain that $y_2$ is a solution to
     \begin{equation}
     \label{at+bt+ct}
          \widetilde a(y) + \widetilde b(y)+\widetilde c(y)=0.
     \end{equation}
On the other hand, it is straightforward that $1$ is a solution to \eqref{at+bt+ct}, because this is equivalent with the equality $\sum_{-1\leq i,j\leq 1}p_{i,j}=1$. Using the root-coefficient relationships, we obtain that the other root is $c(1)/a(1)\geq 1$. This contradicts the fact that $y_2<1$, which is proved in \cite[Lemma 2.3.9]{FIM}. This means that $X_\ell(y_2)\neq 1$. To prove that $X_\ell(y_2)>1$, it is enough to do this for one walk, using the continuity of the different quantities with respect to  the parameters $\{p_{i,j}\}_{-1\leq i,j\leq 1}$. This is easily done for, for example, walks with $p_{1,0}+p_{0,1}+p_{-1,0}+p_{0,-1}=1$.
\end{proof}

{\noindent {\bf Remark.}} The expression of $\beta_0$ given in \eqref{def_beta_0} can be considerably simplified in some cases: see \cite[Remark 16]{KuRa} and \cite[Theorem 3]{Ra} for a list of cases where the conformal mapping $w$ takes a particular simple expression (those simple cases are particular instances of walks for which a certain group of automorphisms in the sense of Malyshev \cite{MAL} is finite). As an example, for walks with $p_{1,0}+p_{0,1}+p_{-1,0}+p_{0,-1}=1$, we have
     \begin{equation}
     \label{ex_w_group_4}
          w(t)=\frac{(t-x_1)(t-x_4)}{(t-x_2)(t-x_3)}.
     \end{equation}
As another example, for walks with $p_{-1,1}+p_{1,0}+p_{0,-1}=1$ as in Figure \ref{Tandem}, we have
     \begin{equation*}
          w(t)=\frac{t}{(t-x_2)(t-[p_{-1,1}p_{0,-1}/({p_{1,0}^2x_2})]^{1/2})^2}.
     \end{equation*}
\medskip

{\noindent {\bf Remark.}} The quantity $B(j_0)$ in \eqref{def_Bj0} must be positive, since it governs the asymptotic behavior of the hitting probabilities, see Theorem \ref{theorem_Godreche}. As we shall see now, we can actually also prove directly the positivity of $B(j_0)$. First, since both \eqref{drift_neg} and \eqref{drift_neg_zero} imply $d'(x_2)<0$, we have that $B(j_0)>0$ if and only if $\beta_0<0$---the latter inequality is not obvious from \eqref{def_beta_0}. In fact, since $\beta_0\neq 0$ (see the proof of Theorem \ref{main_theorem_non-zero} and Lemma \ref{expression_B}), we can check that $\beta_0<0$ for one particular walk, by continuity of $\beta_0$ with respect to the parameters $\{p_{i,j}\}_{-1\leq i,j\leq 1}$. This is easily done with the expression \eqref{ex_w_group_4}.

\begin{proof}[Proof of Theorem \ref{theorem_Godreche}]
The proof of Theorem \ref{theorem_Godreche} is very similar to that of Theorem \ref{main_theorem_non-zero}. First, $X([y_1,y_2])$ is an ellipse \cite[Theorem 6.3.1]{FIM}, and classical expressions for the conformal gluing function are available \cite[Equation (6.3.1)]{FIM}. We then transform this mapping by a linear transformation, in order to have a pole at $x_2$. We call this new function $w$. Next, the integral representation of $h^{i_0,j_0}$ in Theorem \ref{expression_h} remains to hold (see the four steps of Subsection \ref{App_exp_f}). In particular, the expansion of $h^{i_0,j_0}(1)$ given in \eqref{ded_non-zero} is still true. As for the constants, $\rho=x_2$, $\beta_0$ is as in \eqref{def_beta_0} and $C(j_0)$ has exactly the same expression as $B(j_0)$ in \eqref{def_Bj0}.
\end{proof}


\section{Continuum limit for the zero-drift case}\label{driftcon}
\setcounter{equation}{0}
In this section we derive the continuum limit for the zero-drift random walks that satisfy \ref{small_jumps} and \eqref{drift_zero}. As it turns out, the limiting process is much easier to analyze than the original random walk, and in particular, the hitting probability in \eqref{mainprob} is obtained in a straightforward manner. In what follows, we write $\sum_{i,j}$ for $\sum_{-1\leq i,j\leq 1}$, $\mathcal{H}$ for $\mathbb{R}_+$ and $\mathcal{V}$ for $i\mathbb{R}_+$ (hereafter, we shall identify $\mathbb R^2$ and $\mathbb C$). Moreover, let $\Rightarrow$ denote convergence in distribution.
\begin{thm}\label{thm_generator}
     \begin{equation*}
          n^{-1/2}(X(\lfloor n t\rfloor),Y(\lfloor n t\rfloor)_{t\in \mathbb{R}_+}\Rightarrow (X^c,Y^c)=(X_t^c,Y_t^c)_{t\in \mathbb{R}_+}, \quad n\to\infty
     \end{equation*}
     with the generator of $(X^c,Y^c)$ given by
     \begin{equation}
     \label{generator}
          L=\frac{1}{2}\Big[\textstyle \sum_{i,j}i^2 p_{i,j}
          \dfrac{\partial^{2}}{\partial x^2}
          +
          2 \sum_{i,j}i j p_{i,j}
          \dfrac{\partial^{2}}{\partial x \partial y}
          +
          \sum_{i,j}j^2 p_{i,j}
          \dfrac{\partial^{2}}{\partial y^2}\Big].
     \end{equation}
\end{thm}
\begin{proof}
The generator \eqref{generator} simply follows from studying the limit behavior of
     \begin{equation*}
          \frac{1}{\epsilon^2}\big[\textstyle \sum_{i,j}p_{i,j}f(x+i \epsilon,y+j \epsilon)-f(x,y)\big]
     \end{equation*}
as $\epsilon\to 0$, for any function $f$ regular enough; see \cite{EtKu}, see also \cite[Section 5]{AIM}. Indeed, we have the Taylor expansion
     \begin{equation*}
          \textstyle \sum_{i,j}p_{i,j}f(x+i \epsilon,y+j \epsilon)=f(x,y)+\epsilon D f(x,y)+ \epsilon^2 Lf(x,y)+O(\epsilon^3),
     \end{equation*}
where $L$ is given by \eqref{generator} and
     \begin{equation*}
          D=(\widetilde a(1)-\widetilde c(1))\dfrac{\partial}{\partial x}+( a(1)- c(1))\dfrac{\partial}{\partial y}. 
     \end{equation*}
Further, because of our assumption \eqref{drift_zero} on the drift, the operator $D$ is zero. 
\end{proof}

Let
     \begin{equation}
     \label{def_beta}
          \beta=\bigg[\frac{\sum_{i,j}j^2 p_{i,j}}{\sum_{i,j}i^2 p_{i,j}}\bigg]^{1/2}
          =\bigg[\frac{a(1)+c(1)}{\widetilde a(1)+\widetilde c(1)}\bigg]^{1/2}
     \end{equation}
with $a$, $c$, $\widetilde a$ and $\widetilde c$ as in \eqref{def_a_b_c}.
Investigating the continuum limit leads to the following result.
\begin{prop}
\label{absorption_continuum}
\begin{equation*}
     \mathbb{P}_{(x,y)}[(X^c,Y^c) \textnormal{ hits } \mathcal{V} \textnormal{ before } \mathcal{H}^*]
     =\frac{1}{\theta} \arctan \bigg(
     \frac{\sin (\theta)y}{ \beta x +\cos(\theta)y}\bigg),
\end{equation*}
with $\theta$  as in
\eqref{exp_theta}.
\end{prop}
From Proposition \ref{absorption_continuum} the following asymptotic result can be easily extracted.
\begin{cor}
\label{asymptotic_absorption_continuum}
For fixed $y>0$,
\begin{equation}\label{ascon}
     \mathbb{P}_{(x,y)}[(X^c,Y^c) \textnormal{ hits }   \mathcal{V} \textnormal{ before } \mathcal{H}^*]
     \sim \frac{\sin(\theta)}{\beta \theta}\frac{y}{x}, \qquad x\to \infty.
\end{equation}
\end{cor}

The next result shows that the constant term in \eqref{ascon} corresponds with the constant term in the asymptotic expression for the original zero-drift random walk.
\begin{prop}
\label{eq_con}
Let $A$ be the constant in Theorem \ref{main_theorem_zero}. Then
     \begin{equation*}
          A=\frac{[-d''(1)]^{1/2}}{2^{3/2}\theta a(1)}=\frac{\sin(\theta)}{\beta \theta}.
     \end{equation*}
\end{prop}
\begin{proof}
Using \eqref{exp_theta} and \eqref{def_beta}, it is enough to show that
     \begin{equation}
     \label{eqts}
          \frac{[-d''(1)]^{1/2}}{2^{3/2}}=\sin(\theta)a(1)\left[\frac{\widetilde{a}(1)+\widetilde{c}(1)}{ a(1)+ c(1)}\right]^{1/2}
          =\sin(\theta)[a(1)\widetilde a(1)]^{1/2},
     \end{equation}
where we have used that $a(1)=c(1)$ and $\widetilde a(1)=\widetilde c(1)$, due to \eqref{drift_zero}. Since
     \begin{equation*}
          \sin(\theta)= \pm\left( 1-\frac{[\sum_{i,j} i j p_{i,j}]^2}{[\sum_{i,j}i^2 p_{i,j}]\cdot [\sum_{i,j}j^2 p_{i,j}]}\right)^{1/2}
                      = \pm\left(1-\frac{[\sum_{i,j} i j p_{i,j}]^2}{4a(1)\widetilde a(1)}\right)^{1/2},
     \end{equation*}
taking squares at both sides of \eqref{eqts} yields
     \begin{equation*}
          \frac{-d''(1)}{8}=\left(1-\frac{[\sum_{i,j} i j p_{i,j}]^2}{4a(1)\widetilde a(1)}\right)a(1)\widetilde a(1),
     \end{equation*}
so that it suffices to prove that
     \begin{equation}
     \label{eqtp}
          \frac{-d''(1)}{2}= 4 a(1)\widetilde a(1) - [\textstyle\sum_{i,j} i j p_{i,j}]^2.
     \end{equation}
Combining  \eqref{def_a_b_c} and \eqref{def_d} yields
     \begin{equation*}
          d''(1)=2b(1)b''(1)+2b'(1)^2-8a'(1)c'(1)-4a''(1)c(1)-4a(1)c''(1).
     \end{equation*}
Since $a(1)=c(1)$ (see \eqref{drift_zero}) and $a(1)+b(1)+c(1)=0$ (because $\sum_{i,j}p_{i,j}=1$), we obtain
     \begin{equation*}
          d''(1)=2b'(1)^2-8a'(1)c'(1)-4a(1)[a''(1)+b''(1)+c''(1)].
     \end{equation*}
Using that $a'(1)+b'(1)+c'(1)=0$ (the fact that the drifts are zero indeed implies that
the kernel $K(x,1)$ has a root of order two at $1$, see Subsection \ref{Notations}), we get
     \begin{equation*}
          d''(1)=2[a'(1)-c'(1)]^2-8a(1)[p_{1,1}+p_{1,0}+p_{1,-1}]=2[a'(1)-c'(1)]^2-8a(1)\widetilde a(1),
     \end{equation*}
from which \eqref{eqtp} is an immediate consequence.
\end{proof}

\begin{proof}[Proof of Proposition \ref{absorption_continuum}]
As in \cite[Section 5]{AIM} we introduce
     \begin{equation}
     \label{def_phi}
          \phi(x,y)=\alpha (\beta x+\cos (\theta) y, \sin(\theta) y),
     \end{equation}
where we have set
     \begin{equation}
     \label{def_alpha}
          \alpha=\frac{1}{[\sin(\theta)\sum_{ i,j}j^2 p_{i,j}]^{1/2}}.
     \end{equation}
The motivation of this definition is twofold (for a proof of the facts below, we refer to \cite[Part 2]{AIM}):
\begin{itemize}
     \item The covariance of $\phi(X^c,Y^c)$ is described by the identity matrix, so that $\phi(X^c,Y^c)$ lies in the domain of attraction of the standard Brownian motion;
     \item The function $\phi$ satisfies
     \begin{equation}
     \label{transformation_laplacian}
          L(g\circ \phi)=\frac{1}{2}\Delta g \circ \phi,
          \qquad \Delta=\frac{\partial^{2}}{\partial u^2}+\frac{\partial^{2}}{\partial v^2}.
     \end{equation}
\end{itemize}

Introduce the notation $(U^c,V^c)=\phi(X^c,Y^c)$. Using \eqref{def_phi} and the usual identification between $\mathbb R^2$ and $\mathbb C$, this random process lives in the cone
     \begin{equation}
     \label{def_cone}
          \{\rho \exp(i \omega): 0\leq \rho<\infty,\ 0\leq \omega\leq \theta\}.
     \end{equation}
Another property of the process $(U^c,V^c)$ is that (below, $\exp(i \theta)\mathcal{H}$ stands for the rotation by an angle $\theta$ of the horizontal axis $\mathcal{H}$)
     \begin{equation}
     \label{transpro}
          \mathbb{P}_{(x,y)}[(X^c,Y^c) \textnormal{ hits }   \mathcal{V} \textnormal{ before } \mathcal{H}^*]=
          \mathbb{P}_{\phi(x,y)}[(U^c,V^c) \text{ hits } \exp(i \theta)\mathcal{H} \textnormal{ before } \mathcal{H}^*],
     \end{equation}
and we shall now determine
     \begin{equation}
     \label{probab_tbf}
          \mathbb{P}_{(u,v)}[(U^c,V^c) \text{ hits } \exp(i \theta)\mathcal{H} \textnormal{ before } \mathcal{H}^*].
     \end{equation}
Let us first prove that the probability \eqref{probab_tbf} is harmonic inside of the
cone \eqref{def_cone}. For this, note that for $(i_0,j_0)\in\mathbb{Z}_+^2$,
     \begin{equation*}
          \textstyle\sum_{i,j}p_{i,j}\mathbb{P}_{(i_0+i,j_0+j)}[(X,Y) \textnormal{ hits } \mathcal{V} \textnormal{ before } \mathcal{H}^*]
          =\mathbb{P}_{(i_0,j_0)}[(X,Y) \textnormal{ hits } \mathcal{V} \textnormal{ before } \mathcal{H}^*].
     \end{equation*}
As a consequence, on the positive quadrant we have
     \begin{equation*}
          L\big(\mathbb{P}_{(x,y)}[(X^c,Y^c) \textnormal{ hits } \mathcal{V} \textnormal{ before } \mathcal{H}^*]\big)=0,
     \end{equation*}
with $L$ as in \eqref{generator}. Thanks to \eqref{transformation_laplacian}, the probability
\eqref{probab_tbf} is then harmonic in the cone \eqref{def_cone}. Therefore, we can search for
it as the imaginary part of a certain holomorphic function $H$---unique, up to some real additive
constants. In other words,
     \begin{equation}
     \label{searching_A}
          \mathbb{P}_{(u,v)}[(U^c,V^c) \text{ hits } \exp(i \theta)\mathcal{H} \textnormal{ before } \mathcal{H}^*]=\Im[H(u+i v)].
     \end{equation}

To find this function $H$, we shall exploit the fact that we know its boundary values. Since it is impossible
(resp.\ certain) for the process $(X^c,Y^c)$ to be absorbed when starting from
$\mathcal{H}^*$ (resp.\ $\mathcal{V}$), one must have
     \begin{equation*}
          \Im[H(z)]=\left\{\begin{array}{lll}
          0&\text{if}& z\in \mathcal{H}^*,\\
          1&\text{if}& z\in \exp(i \theta)\mathcal{H}^*.
          \end{array}\right.
     \end{equation*}
A suitable choice for $H$ is
     \begin{equation}
     \label{choice_A}
          H(z)=\frac{1}{\theta}\log(z),
     \end{equation}
where $\log$ refers to the principal determination of the logarithm: this means, see \cite{JS}, that for $z\in\mathbb{C}\setminus (-\infty,0]$ we have
\begin{equation*}\label{nlog}
\log(z)=\log(|z|)+i \arg(z),
 \end{equation*}
 where in the right-hand side, $\log$ stands for the classical logarithm function on $(0,\infty)$.

For \eqref{searching_A} and \eqref{choice_A} we deduce that
     \begin{equation*}
          \mathbb{P}_{(u,v)}[(U^c,V^c) \text{ hits } \exp(i \theta)\mathcal{H} \textnormal{ before } \mathcal{H}^*]
          =\frac{1}{\theta}\arctan\Big(\frac{v}{u}\Big).
     \end{equation*}
Thanks to \eqref{def_phi} and \eqref{transpro}, the proof of Proposition \ref{absorption_continuum}
is then completed.
\end{proof}

\section*{Acknowledgments}
We thank S.C.F.\ van Opheusden and F.\ Redig for inspiring discussions. We also thank an anonymous referee for useful comments and suggestions.
K.\ Raschel would like to thank EURANDOM for providing wonderful working conditions during a visit in the year 2011. His work was partially supported by CRC 701, Spectral Structures and Topological Methods in Mathematics at the University of Bielefeld. J.S.H.\ van Leeuwaarden is supported by an ERC Starting Grant.

\end{document}